\input ./style/arxiv-vmsta.cfg
\documentclass[numbers,compress,v1.0.1]{vmsta}
\usepackage{vtexbibtags}
\usepackage{mathtools}
\usepackage[shortlabels]{enumitem}

\volume{6}
\issue{1}
\pubyear{2019}
\firstpage{57}
\lastpage{79}
\aid{VMSTA122}
\doi{10.15559/18-VMSTA122}
\articletype{research-article}

\startlocaldefs

\allowdisplaybreaks

\newtheorem{theorem}{Theorem}
\newtheorem{lemma}{Lemma}
\newtheorem{proposition}{Proposition}
\newtheorem{corollary}{Corollary}
\theoremstyle{definition}
\newtheorem{definition}{Definition}
\theoremstyle{remark}
\newtheorem{remark}{Remark}

\newcommand{\F}{\mathcal{F}}
\newcommand{\FF}{\mathbb{F}}
\newcommand{\E}{\mathsf{E}}
\newcommand{\Prob}{\mathsf{P}}
\newcommand{\R}{\mathbb{R}}
\newcommand{\N}{\mathbb{N}}
\newcommand{\B}{\mathcal{B}}
\newcommand{\Beta}{\mathrm{B}}

\newcommand{\oD}{\overline{D}}
\newcommand*{\set}[1]{\left\{#1\right\}}
\newcommand*{\abs}[1]{\left\lvert#1\right\rvert}
\newcommand*{\norm}[1]{\left\lVert#1\right\rVert}

\endlocaldefs
\DeclareMathOperator{\Div}{div}

\begin{document}

\begin{frontmatter}
\pretitle{Research Article}

\title{Existence and uniqueness of mild solution to fractional
stochastic heat equation}

\author{\inits{K.}\fnms{Kostiantyn}~\snm{Ralchenko}\ead
[label=e1]{k.ralchenko@gmail.com}}
\author{\inits{G.}\fnms{Georgiy}~\snm{Shevchenko}\thanksref{cor1}\ead
[label=e2]{zhora@univ.kiev.ua}\orcid{0000-0003-1047-3533}}
\thankstext[type=corresp,id=cor1]{Corresponding author.}
\address{Department of Probability Theory, Statistics and Actuarial
Mathematics, \institution{Taras~Shevchenko National University of
Kyiv}, 64,~Volodymyrs'ka~St.,~01601~Kyiv,~\cny{Ukraine}}

\markboth{K. Ralchenko, G. Shevchenko}{Existence and uniqueness of mild solution to fractional stochastic heat equation}

\begin{abstract}
For a class of non-autonomous parabolic stochastic partial
differential equations defined on a bounded open subset $D \subset\R
^d$ and driven by an $L^2(D)$-valued fractional Brownian motion with
the Hurst index $H>1/2$,
a new result on existence and uniqueness of a mild solution is established.
Compared to the existing results,
the
uniqueness in a fully
nonlinear case is shown, not assuming the coefficient in front of the noise to
be affine. Additionally, 
the existence of moments for the solution is established.
\end{abstract}
\begin{keywords}
\kwd{Fractional Brownian motion}
\kwd{stochastic partial differential equation}
\kwd{mild solution}
\kwd{Green's function}
\end{keywords}
\begin{keywords}[MSC2010]%
\kwd{60H15}
\kwd{35R60}
\kwd{35K55}
\kwd{60G22}
\end{keywords}

\received{\sday{4} \smonth{8} \syear{2018}}
\revised{\sday{23} \smonth{10} \syear{2018}}
\accepted{\sday{23} \smonth{10} \syear{2018}}
\publishedonline{\sday{12} \smonth{12} \syear{2018}}
\end{frontmatter}

\section{Introduction}

In this paper we study an initial--Neumann boundary value problem for
the following non-autonomous stochastic partial differential equation
of parabolic type\index{parabolic type} in a cylinder domain $D\times[0,T]$, driven by an
infinite-dimensional fractional noise:\index{fractional noise}
\begin{align}
du(x,t) &{}= \bigl(\Div \bigl(k(x,t)\nabla
u(x,t) \bigr)+f \bigl(u(x,t) \bigr) \bigr)dt + h \bigl(u(x,t) \bigr)
W^H(x,dt),\notag
\\
(x,t) &{}\in D\times(0,T],\notag
\\
u(x,0) &{}= \varphi(x), \quad x \in D,\notag
\\
\frac{\partial u(x,t)}{\partial n(k)}&{}=0, \quad(x,t)\in\partial D \times(0,T]. \label{eq:spde} %
\end{align}
Here
$D\subset\R^d$ is a bounded domain with the boundary $\partial D$
of class $C^{2+\beta}$ with some $\beta\in(0,1)$,
$W^H$ is an $L^2(D)$-valued fractional Brownian motion\index{fractional Brownian motions} with the Hurst
index $H\in(\frac{1}2,1)$,
$k=\set{k_{i,j}}\colon\oD\times[0,T]\to\R^{d\times d}$ is a
matrix-valued field,
$n(k)(x)\coloneqq k(x,t)n(x)$ denotes the conormal vector-field, and
the last relation in \eqref{eq:spde} refers to the conormal derivative
of $u$ relative to $k$, that is
\[
\frac{\partial u(x,t)}{\partial n(k)} =\sum_{i,j=1}^d
k_{i,j}(x,t) n_i(x) \frac{\partial}{\partial x_j}u(x,t),
\]
$n(x)\in\R^d$ is an outer normal vector to $\partial D$.

Equations similar to \eqref{eq:spde} were studied extensively in
literature, so we will mention only several most relevant articles
here. The articles \cite{grecksch-anh} and \cite{hu2001} considered
heat equations with additive and multiplicative fractional noise,\index{fractional noise}
respectively. The articles \cite{chenetal,Maslowski03,Veraar10} are
devoted to general non-linear evolution equations with fractional
noise;\index{fractional noise} however, the equations are considered in some functional spaces,
and the assumptions on the coefficients imposed there do not cover
general nonlinear equations of the form \eqref{eq:spde}.
The problem \eqref{eq:spde} was considered in \cite{NV06} and then in
\cite{SSV09}, where the notions of \emph{variational} and \emph{mild}
solutions\index{mild solution} were introduced. The article \cite{NV06} established the
existence of a variational solution\index{variational solution} to this equation, but the
uniqueness was shown under the assumption that the function $h$ is
affine. In \cite{SSV09}, it was shown that a variational solution\index{variational solution} to
\eqref{eq:spde} is a mild solution\index{mild solution} too, however, the uniqueness, both
in the variational and in the mild sense, was established also under
the assumption that $h$ is affine, moreover, a rather restrictive
assumption $H>\frac{d+1}{d+2}$ on the Hurst exponent was imposed.

Our goal is to extend the uniqueness results of \cite{SSV09} to the
case of arbitrary $H\in(\frac{1}2,1)$ and non-affine $h$.
Specifically, we prove the uniqueness of a mild solution,\index{mild solution} assuming that
$h$ and its derivative $h'$ are Lipschitz continuous. Since the
existence of a variational solution\index{variational solution} is known from \cite{NV06}, and \cite
{SSV09} established that each variational solution\index{variational solution} is a mild solution,\index{mild solution}
we get existence and uniqueness of a variational solution\index{variational solution} too, thus
finally answering a question posed in \cite{NV06}. We also show that
the solution to \eqref{eq:spde} has finite moments of any order.

It is worth to mention that a similar uniqueness result holds in the
case where the function $h$ in front of $W^H$ depends on $t$
sufficiently regularly, say, H\"older continuous with exponent greater
than $1/2$. However, since our main reference for existence results are
the articles \cite{NV06,SSV09}, in which $h$ is assumed to be
independent of $t$, we will follow this assumption.

The paper is organized as follows.
In Section~\ref{sec:prelim}, we formulate the main hypotheses, and give
the definition of a mild solution\index{mild solution} and basic facts on an $L^2(D)$-valued
fractional Brownian\index{fractional Brownian process} process and stochastic integration\index{stochastic integration} with respect to it.
Section~\ref{sec:Green} contains auxiliary results concerning the
parabolic\index{parabolic Green's function} Green's function.
In Section~\ref{sec:apriori}, we give a priori upper bounds for mild solutions.\index{mild solution}
Finally, the main result on existence and uniqueness of a mild solution\index{mild solution}
is proved in Section~\ref{sec:mild}.

\section{Preliminaries}
\label{sec:prelim}
This section is devoted to the precise statement of the problem \eqref{eq:spde}.
We introduce necessary notation, give the definition of a mild
solution,\index{mild solution} and formulate the assumptions for its existence and uniqueness.

\subsection{Notational conventions}

Throughout the article, $\abs{\,\cdot\,}$ will denote absolute value of
a number, Euclidean norm of a vector or operator norm of a matrix;
exact meaning will always be clear from the context. We will use the
symbol $C$ for a generic constant, the precise value of which is not
important and may vary between different equations and inequalities.

\subsection{Norms and spaces}
Let $\norm{\,\cdot\,}_2$ and $\norm{\,\cdot\,}_\infty$ be the norms in
$L^2(D)$ and $L^\infty(D)$, respectively.
Denote for $\alpha\in(0,1)$,
$u\colon D\times[0,T]\to\R$ and $t\in[0,T]$
\begin{align*}
\norm{u}_{\alpha,1,t}& \coloneqq \sup_{x\in D}\sup
_{s\in[0,t]}\int_0^s
\frac{\abs{u(x,s)-u(x,v)}}{(s-v)^{\alpha+1}}\,dv,
\\
\norm{u}_{\alpha,\infty,t} &\coloneqq \sup_{s\in[0,t]} \norm{u(
\cdot,s)}_\infty+ \norm{u}_{\alpha,1,t},
\\
%
\norm{u}_{\alpha,2,t} &\coloneqq \Biggl( \sup_{s\in[0,t]}\norm{u(
\cdot,s)}_2^2 + \int_0^t
\Biggl( \int_0^s \frac{\norm{u(\cdot,s)-u(\cdot
,v)}_2}{(s-v)^{\alpha+1}}\,dv
\Biggr)^2 ds \Biggr)^{1/2}.
\end{align*}
Denote by $\B^{\alpha,2} (0,T;L^2(D) )$ the Banach space of
measurable mappings $u\colon D\times[0,T]\to\R$ such that $\norm
{u}_{\alpha,2,T}^2<\infty$.

Let $H^1(D)$ be the Sobolev space of functions $f\colon D\to\R$
equipped with the norm
$\norm{f}_{1,2} =  (\norm{f}^2_2+\norm{\nabla f}^2_2 )^{1/2}$.
Also introduce the space
$L^2 (0,T;H^1(D) )$ of measurable mappings $u\colon[0,T]\to
H^1(D)$ such that
\[
\int_0^T\norm{u(t)}_{1,2}^2
\,dt = \int_0^T \bigl(\norm{u(t)}^2_2+
\norm {\nabla u(t)}^2_2 \bigr) dt < \infty.
\]

For $f\colon[0,T]\to\R$ and $\alpha\in(0,1)$ define a seminorm
\[
\norm{f}_{\alpha,0,t} = \sup_{0\le u< v < t} \Biggl(\frac{\abs
{f(v)-f(u)}}{(v-u)^{1-\alpha}}
+ \int_u^v \frac{\abs
{f(u)-f(z)}}{(z-u)^{2-\alpha}}\,dz \Biggr).
\]

\subsection{Assumptions on the coefficients and on the initial value}
\begin{enumerate}[label=\bf(A\arabic*),ref=\rm(A\arabic*)]
\item\label{(A1)}
The coefficients $k_{ij}$ satisfy the following assumptions:
\begin{enumerate}[(i)]
\item
$k_{i,j}=k_{j,i}$ for all $i,j=1,\ldots,d$;
\item
$k_{i,j} \in C^{\beta,\beta'}(\oD\times[0,T])$
for some $\beta'\in(\frac{1}2,1]$ and for all $i,j=1,\ldots,d$;
\item
$\frac{\partial}{\partial x_l} k_{i,j} \in C^{\beta,\beta/2}(\oD\times
[0,T])$ for all $i,j,l=1,\ldots,d$;
\item
there exists $\underline k>0$ such that
\[
\sum_{i,j=1}^d k_{i,j}(x,t)q_iq_j
\ge\underline k \abs{q}^2,
\]
for all $x\in\oD$, $t\in[0,T]$, $q\in\R^d$;
%
\item
$\displaystyle(x,t)\mapsto\sum_{i=1}^dk_{i,j}(x,t)n_i(x) \in C^{1+\beta
,(1+\beta)/2}(\partial D\times[0,T])$ for each $j$;
\item
the conormal vector-field
$(x,t)\mapsto n(k)(x,t)= k(x,t)n(x)$
is outward pointing, nowhere tangent to $\partial D$ for every $t$.
\end{enumerate}

\item\label{(A2)}
The initial condition $\varphi\in C^{2+\beta}(\oD)$ satisfies the
conormal boundary condition relative to $k$.

\item\label{(A3)}
$f,h,h'\colon\R\to\R$ are globally Lipschitz continuous functions.
\end{enumerate}

\begin{remark}
The global Lipshitz assumption implies that $f$ and $h$ are of linear growth:\index{linear growth}
\begin{equation}
\label{eq:lingrow} |f(x)| + |h(x)| \le C(1+|x|).
\end{equation}
It is worth to mention that all results of the article can be proved
assuming linear growth\index{linear growth} and only local Lipschitz continuity\index{Lipschitz continuity} of $f$ and
$h'$ with some extra technical work. We decided to impose the global
Lipschitz continuity\index{Lipschitz continuity} assumption for the sake of simplicity and because
it does not lead to a considerable loss of generality.
\end{remark}

\subsection{\texorpdfstring{$L^2(D)$}{L\texttwosuperior(D)}-valued fractional Brownian\index{fractional Brownian process} process and stochastic integration\index{stochastic integration} with respect to it}

Let us briefly recall the definition of an $L^2(D)$-valued fractional
Brownian\index{fractional Brownian process} process and the corresponding stochastic integral, introduced
in \cite{Maslowski03}.
Assume that $\{\lambda_j, j\,{\in}\,\N\}$ is a sequence of positive real
numbers and
$\set{e_j,j\in\N}$ is an orthonormal basis of $L^2(D)$
such \xch{that}{that that}
\begin{enumerate}[resume*]
\item\label{(A4)}
$\displaystyle\sup_j\norm{e_j}_\infty<\infty
\quad\text{and}\quad
\sum_{j=1}^\infty\lambda_j^{1/2}<\infty$.
\end{enumerate}

Let $(\varOmega,\F,\Prob)$ be a complete probability space.
For a fixed $T>0$ let $\FF=\set{\F}_{t\in[0,T]}$ be a filtration
satisfying the standard assumptions.
Let
$B^H_j=\{B^H_j(t),\allowbreak t\ge0\}$, $j\in\N$,
be a sequence of one-dimensional, independent fractional Brownian
motions\index{fractional Brownian motions} with the Hurst parameter $H\in(1/2,1)$, defined on $(\varOmega,\F
,\FF,\Prob)$ and starting at the origin.
Following \cite{Maslowski03}, define $L^2(D)$-valued fractional
Brownian\index{fractional Brownian process} process $W^H=\set{W^H(\cdot,t),t\ge0}$ by
\begin{equation*}
W^H(\cdot,t) = \sum_{j=1}^\infty
\lambda_j^{1/2} e_j(\cdot) B^H_j(t),
\end{equation*}
where the series converges a.\,s.\ in $L^2(D)$.

In this article we consider a pathwise stochastic integration\index{pathwise stochastic integration} with
respect to $W^H$ in the fractional (generalized Lebesgue--Stieltjes)
sense. Alternatively, one can look at the so-called Skorokhod
(white-noise) integral. However, with the Skorokhod definition,\index{Skorokhod definition} it is
difficult to solve even stochastic ordinary differential equations, see
e.g. \cite{hu2018}.

Fix $\alpha\in(1-H,1/2)$.
Let $\varPhi= \set{\varPhi(t), t\in[0,T]}$ be an adapted stochastic process
taking values in the space of linear bounded operators on $L^2(D)$ such that
\[
\sup_{j\in\N} \int_0^T\!
\Biggl( \frac{\norm{\varPhi(t)e_j}_2}{t^\alpha} + \int_0^t
\frac{\norm{(\varPhi(t)-\varPhi(s))e_j}_2}{(t-s)^{\alpha+1}}ds \Biggr)dt <\infty.
\]
Following \cite{Maslowski03} (see also \cite{NV06,SSV09}), we introduce
the integral with respect to an $L^2(D)$-valued fractional Brownian\index{fractional Brownian process}
process by
\[
\int_a^b \varPhi(s) \,dW^H(s)
\coloneqq\sum_{j=1}^\infty\lambda_j^{1/2}
\int_a^b \varPhi(s) e_j
\,dB^H_j(s),
\]
where the integrals with respect to $B^H_j$, $j\in\N$, are understood
as pathwise generalized Lebesgue--Stieltjes integrals.
Such integrals are defined in terms of fractional derivatives, the
detailed exposition of this approach can be found,
e.\,g., in the book \cite[Section 2.1]{mishura}.
We mention only that under above assumptions, the generalized
Lebesgue--Stieltjes integral $\int_a^b \varPhi(s) e_j \,dB^H_j(s)$ is well
defined and
admits the bound
\begin{align}
&\abs{\int_a^b \varPhi(s) e_j \,dB^H_j(s)}\notag\\
&\quad{}\le C_\alpha\norm{B^H_j}_{\alpha,0,b} \int
_a^b \Biggl(\frac{\abs{\varPhi(s) e_j}}{(s-a)^\alpha} +
\int_a^s\frac{\abs{ (\varPhi(s)-\varPhi(v) ) e_j}}{(s-v)^{\alpha+1}}\, dv \Biggr)ds
\label{eq:bound-int}
\end{align}
for some constant $C_\alpha>0$.

\subsection{Mild solution\index{mild solution}}
Following \cite{SSV09},
we understand a solution to the problem \eqref{eq:spde} in a mild sense.
Its definition uses the notion of the parabolic\index{parabolic Green's function} Green's function
$G(x,t,y,s)$, $x,y\in\oD$, $0\le s<t\le T$, associated with the
principal part of \eqref{eq:spde} (see, e.\,g., \cite
{EI70,Eidelman98,friedman1969,ladyzhenskaya1968}).
For every $(y,s)\in D\times(0,T]$, $G(x,t,y,s)$ is a classical solution
to the linear initial-boundary value problem
\begin{align}
\partial_t G(x,t;y,s) &= \Div
\bigl(k(x,t)\nabla_x G(x,t;y,s) \bigr), \quad(x,t) \in D\times(0,T],
\notag
\\
\frac{\partial G(x,t;y,s)}{\partial n(k)}&=0, \quad(x,t)\in\partial D \times(0,T],
\label{eq:pde} %
\end{align}
with
\[
\int_D G(\cdot,s;y,s) \varphi(y)\,dy := \lim
_{t\downarrow s} \int_D G(\cdot,t;y,s) \varphi(y)
\,dy = \varphi(\cdot).
\]
In the next section we consider the properties of $G$ in detail.

\begin{definition}[\cite{SSV09}]
Fix $H\in(1/2,1)$ and $\alpha\in(1-H,1/2)$.
An $L^2(D)$-valued random field $\set{u(\cdot,t),t\in[0,T]}$ is a \emph
{mild solution\index{mild solution}} to the problem \eqref{eq:spde} if the following two
conditions are satisfied:
\begin{enumerate}[(1)]
\item$u\in L^2 (0,T;H^1(D) ) \cap\B^{\alpha,2}
(0,T;L^2(D) )$ a.\,s.
\item
The relation
\begin{align}
u(\cdot,t) &= \int_D G(\cdot,t;y,0)\varphi(y)\,dy + \int
_0^t\int_D G(
\cdot,t;y,s)f\bigl(u(y,s)\bigr)\,dy\,\xch{ds}{ds+{}} \notag
\\
&\quad+ \sum_{j=1}^\infty
\lambda_j^{1/2} \int_0^t
\int_D G(\cdot ,t;y,s)h\bigl(u(y,s)\bigr)e_j(y)
\,dy\,dB^H_j(s) \label{eq:mild}
\end{align}
holds a.\,s.\ for every $t\in[0,T]$ as an equality in $L^2(D)$.
\end{enumerate}
\end{definition}

\section{Properties of Green's function}
\label{sec:Green}
In this section we collect several upper bounds for Green's function
$G$, needed for the proof of the main result.

Denote
\[
\varPhi^C_{t}(x)\coloneqq t^{-d/2} \exp \bigl
\{-C\abs{x}^2/t \bigr\}, \quad t>0,\; x\in\R^d.
\]

It is known from \cite{EI70,Eidelman98} that under assumptions \ref
{(A1)} and \ref{(A2)}
$G$ is a continuous function, twice continuously differentiable in $x$,
once continuously differentiable in $t$.
Moreover, $G$ satisfies the heat kernel estimates
\begin{equation}
\label{eq:dG1} \abs{\partial_x^\mu\partial_t^\nu
G(x,t;y,s)} \le C(t-s)^{-(\abs{\mu}_1+2\nu)/2} \varPhi^C_{t-s}(x-y)
\end{equation}
for $\mu=(\mu_1,\dots,\mu_d)$,
$\mu_1,\dots,\mu_d,\nu\in\N\cup\set{0}$, and
$\abs{\mu}_1 + 2\nu\le2$ with
$\abs{\mu}_1 = \sum_{j=1}^d\mu_j$.
In particular, for $\abs{\mu}_1=\nu=0$, we have
\begin{equation}
\label{eq:gaus} \abs{G(x,t;y,s)} \le C \varPhi^C_{t-s}(x-y).
\end{equation}
The inequality \eqref{eq:gaus} is sometimes called the \emph{Gaussian\index{Gaussian property}
property} of $G$ (\cite{SSV03,SSV09}).

Several important properties of the parabolic\index{parabolic Green's function} Green's function $G$
follow from the fact that it is,
for every $(x,t)\in D\times[0,T]$, a classical solution to the linear
boundary value problem
\begin{align*}
\partial_s G(x,t;y,s) &= -\Div \bigl(k(y,s)\nabla_y
G(x,t;y,s) \bigr), \quad(y,s) \in D\times(0,T], 
\\
\frac{\partial G(x,t;y,s)}{\partial n(k)}&=0, \quad(y,s)\in\partial D \times(0,T],
\end{align*}
dual to \eqref{eq:pde}.
In particular, along with \eqref{eq:dG1} we have also
\begin{equation}
\label{eq:dG2} \abs{\partial_y^\mu\partial_s^\nu
G(x,t;y,s)} \le C(t-s)^{-(\abs{\mu}_1+2\nu)/2} \varPhi^C_{t-s}(x-y)
\end{equation}
for $\abs{\mu}_1 + 2\nu\le2$, and, moreover, the following
convolution formula holds:
\begin{equation}
\label{eq:conv} G(x,t;y,s) = \int_D G(x,t;z,\sigma) G(z,
\sigma;y,s)\,dz \quad\text{for all }\sigma\in(s,t),
\end{equation}
see \cite[Appendix, p.~232--233]{EI70} for the details.\goodbreak

Furthermore, according to Eqs.\ (3.4)--(3.5) from \cite{SSV09},
$G$ satisfies the following inequalities
for all $x,y\in D$ and $\delta\in(\frac{d}{d+2},1)$.
\begin{enumerate}[(i)]
\item
For all $0<r<v<t<T$ and some $t^*\in(r,v)$,
\begin{equation}
\label{eq:deltaG1} \abs{G(x,t;y,v)-G(x,t;y,r)} \le C(t-v)^{-\delta}
(v-r)^{\delta}\, \varPhi ^C_{t-t^*}(x-y).
\end{equation}
\item
For all $0<v<s<t<T$ and some $v^*\in(s,t)$,
\begin{equation}
\label{eq:deltaG} \abs{G(x,t;y,v)-G(x,s;y,v)} \le C(t-s)^\delta(s-v)^{-\delta}
\, \varPhi ^C_{v^*-v}(x-y).
\end{equation}
\end{enumerate}

\begin{lemma}
Under assumptions \ref{(A1)}--\ref{(A2)},
for all $0<r<v<s<t<T$ and for all $x,y\in D$,
\begin{align}
&\abs{G(x,t;y,v)-G(x,s;y,v)-G(x,t;y,r)+G(x,s;y,r)}\notag\\
&\quad{}\le C \int_r^v \int
_s^t (\theta-\tau)^{-2}
\varPhi_{\theta
-\tau}^C (x-y) d\theta\,d\tau. \label{eq:deltaG2}
\end{align}
\end{lemma}

\begin{proof}
Write
\begin{align}
& G(x,t;y,v)-G(x,s;y,v)-G(x,t;y,r)+G(x,s;y,r)
\notag\\*
&\label{eq:deltaG3}\quad{} =\int_s^t \int_r^v
\frac{\partial^2}{\partial\theta\,\partial\tau} G(x,\theta;y,\tau) \,d\tau\,d\theta.
\end{align}
By \eqref{eq:conv}, the equality
\[
\frac{\partial^2}{\partial\theta\,\partial\tau} G(x,\theta;y,\tau) = \int_D
\partial_\theta G(x,\theta;z,\sigma) \partial_\tau G(z,\sigma
;y,\tau)\,dz
\]
holds with $\sigma= \frac{\tau+ \theta}{2}$.
By applying the bounds\index{bounds} \eqref{eq:dG1} and \eqref{eq:dG2} with $\abs{\mu
}_1=0$ and $\nu=1$, we see that
\begin{align*}
\abs{\frac{\partial^2}{\partial\theta\,\partial\tau} G(x,\theta;y,\tau)} &\le \int_D
\abs{\partial_\theta G(x,\theta;z,\sigma) \partial_\tau G(z,
\sigma;y,\tau)}\,dz\,d\sigma
\\
&\le C (\theta-\sigma)^{-1} (\sigma-\tau)^{-1} \int
_D \varPhi^C_{\theta-\sigma}(x-z)
\varPhi^C_{\sigma-\tau}(z-y) \,dz
\\
& \le C (\theta-\tau)^{-2}\varPhi_{\theta-\tau}^C(x-y).
\end{align*}
Combining this bound with \eqref{eq:deltaG3}, we conclude the proof.
\end{proof}

\section{A priori estimates}
\label{sec:apriori}

Fix $H\in(1/2,1)$ and $\alpha\in(1-H,1/2)$.
Let $u$ be a mild solution\index{mild solution} to \eqref{eq:spde}, defined by~\eqref{eq:mild}.
Note that the random variable
\begin{equation*}
\xi_{\alpha,H,T}\coloneqq1+\sum_{j=1}^\infty
\lambda_j^{1/2} \norm {B^H_j}_{\alpha,0,T}
\end{equation*}
is finite a.\,s., see \cite{MRS18}.

The goal of this section is to prove the following result.
\begin{proposition}
\label{pr:1}
Under assumptions \ref{(A1)}--\ref{(A4)},
\[
\norm{u}_{\alpha,\infty,T} \le C \exp\set{C \xi_{\alpha
,H,T}^{1/(1-\alpha)}}.
\]
\end{proposition}

We split the proof of Proposition~\ref{pr:1} into two lemmas.
In Lemma \ref{l:1st} we establish an upper bound for
$\sup_{s\in[0,t]} \sup_{x\in D} \abs{u(x,s)}$.
In Lemma \ref{l:2nd} we obtain similar estimate for
$\norm{u}_{\alpha,1,t}$.

In the calculations below we shall often refer to the following simple formulas:
for all $a>0$, $b>0$, and $0 < v < t$,
\begin{align}
\int_v^t (t-s)^{a-1}
(s-v)^{b-1}\,ds &=C (t-v)^{a+b-1}, \label{eq:beta-1}
\\
\int_0^v (t-s)^{-a-b}
(v-s)^{b-1}\,ds &\le C (t-v)^{-a}, \label{eq:beta-2}
\end{align}
where $C=\Beta(a,b)$, the beta function.\index{beta function}
The formula \eqref{eq:beta-1} follows directly from the definition of
the beta function\index{beta function} by the substitution $z=\frac{s-v}{t-v}$. The
inequality \eqref{eq:beta-2} is obtained by the substitution $z=\frac
{v-s}{t-v}$ as follows:
\begin{align*}
&{}\int_0^v (t-s)^{-a-b}
(v-s)^{b-1}\,ds = (t-v)^{-a} \int_0^{\frac{v}{t-v}}
\frac{z^{b-1}}{(1+z)^{a+b}} \,dz
\\
&\quad{}\le(t-v)^{-a} \int_0^\infty
\frac{z^{b-1}}{(1+z)^{a+b}} \,dz =\Beta(a,b) (t-v)^{-a}.
\end{align*}

Denote for brevity
\[
\norm{u}_s = \sup_{s\in[0,t]}\sup_{x\in D}
\abs{u(x,s)}. %
\]
\begin{lemma}
\label{l:1st}
Under assumptions \ref{(A1)}--\ref{(A4)},
\begin{equation}
\label{eq:apr-term1} \norm{u}_t \le C\xi_{\alpha,H,T} \Biggl(1+\int
_0^t \bigl(\norm{u}_s
(t-s)^{-\alpha} + \norm{u}_{\alpha,1,s} \bigr)ds \Biggr),
\end{equation}
for all $t\in[0,T]$.
\end{lemma}

\begin{proof}
By \eqref{eq:mild},
\begin{equation}
\label{eq:apr-I} \abs{u(x,t)} \le I_\varphi+ I_f +
I_h,
\end{equation}
where
\begin{align*}
I_\varphi&= \int_D \abs{G(x,t;y,0)\varphi(y)}
\,dy,
\\
I_f &= \int_0^t\int
_D \abs{G(x,t;y,s)f\bigl(u(y,s)\bigr)}\,dy\,ds,
\\
I_h &= \sum_{j=1}^\infty
\lambda_j^{1/2} \abs{\int_0^t
\int_D G(x,t;y,s)h\bigl(u(y,s)\bigr)e_j(y)
\,dy\,dB^H_j(s)}.
\end{align*}
By \ref{(A2)}, $\varphi$ is bounded. Therefore, the Gaussian property\index{Gaussian property}
\eqref{eq:gaus} implies that
\begin{equation}
\label{eq:apr-I0} I_\varphi\le C\int_D
\varPhi^C_{t}(x-y)\,dy \le C.
\end{equation}
It follows from the linear growth\index{linear growth} property \eqref{eq:lingrow} that
\begin{equation}
\label{eq:f-growth} \abs{f\bigl(u(y,s)\bigr)} \le C \bigl(1+\abs{u(y,s)}\bigr) \le C
(1+\norm{u}_s).
\end{equation}
Then, applying \eqref{eq:gaus}, we get
\begin{equation}
\label{eq:apr-If} %
\begin{aligned} I_f &\le C \int
_0^t (1+\norm{u}_s )\int
_D \abs{G(x,t;y,s)}\, dy\,ds\le C \int_0^t
(1+\norm{u}_s )\,ds. \end{aligned} %
\end{equation}

Using the bound \eqref{eq:bound-int} for the integrals with respect to
$B^H_j$, we may write
\begin{align*}
I_h &\le C \sum_{j=1}^\infty
\lambda_j^{1/2} \norm{B^H_j}_{\alpha,0,t}
\int_0^t \Biggl( \frac{\int_D \abs{G(x,t;y,s)h(u(y,s))e_j(y)}\,dy}{s^{\alpha}}
\\*
&\quad{}+ \int_0^s \frac{\int_D \abs{G(x,t;y,s)h(u(y,s))
- G(x,t;y,v)h(u(y,v))} \abs{e_j(y)}dy}{(s-v)^{\alpha+1}}\,dv \Biggr)
ds.
\end{align*}
The assumption \ref{(A4)} implies that $\sup_j\abs{e_j(y)}\le C$ for
all $y\in D$.
Therefore,
\begin{equation}
\label{eq:apr-Ih} I_h \le C \xi_{\alpha,H,T} (I_{h1}+I_{h2}+I_{h3}),
\end{equation}
where
\begin{align*}
I_{h1} &= \int_0^t s^{-\alpha}
\int_D \abs{G(x,t;y,s)h\bigl(u(y,s)\bigr)}\,dy\,ds,
\\
I_{h2} &= \int_0^t \int
_0^s (s-v)^{-\alpha-1}\int
_D \abs{G(x,t;y,s)} \abs{h\bigl(u(y,s)\bigr) - h
\bigl(u(y,v)\bigr)} dy\,dv\,ds,
\\
I_{h3} &= \int_0^t \int
_0^s (s-v)^{-\alpha-1}\int
_D \abs{G(x,t;y,s) - G(x,t;y,v)} \abs{h\bigl(u(y,v)\bigr)}
dy\,dv\,ds.
\end{align*}
The term $I_{h1}$ can be estimated similarly to $I_f$, using the linear
growth\index{linear growth} of $h$ and the Gaussian\index{Gaussian property} property of $G$:
\begin{equation*}
I_{h1} \le C \int_0^t (1+
\norm{u}_s )s^{-\alpha}\,ds.
\end{equation*}
Since $\norm{u}_s$ is non-decreasing and $s^{-\alpha}$ is
non-increasing, we can use the rearrangement inequality \cite[Theorem
378]{hardy-littlewood} to obtain
\begin{equation}
\label{eq:apr-Ih1} I_{h1} \le C \int_0^t
(1+\norm{u}_s ) (t-s)^{-\alpha}\,ds.
\end{equation}
By the Lipschitz continuity\index{Lipschitz continuity} of $h$,
\[
I_{h2} \le\int_0^t \int
_D \abs{G(x,t;y,s)} \int_0^s
\frac{\abs{u(y,s) - u(y,v)}}{(s-v)^{\alpha+1}}\,dv\,dy\,ds.
\]
The inner integral can be bounded by $\norm{u}_{\alpha,1,s}$.
Therefore, we get
\begin{equation}
\label{eq:apr-Ih2} I_{h2} \le\int_0^t
\norm{u}_{\alpha,1,s} \int_D \abs{G(x,t;y,s)}\,dy\,ds
\le C \int_0^t \norm{u}_{\alpha,1,s} \,ds.
\end{equation}
In order to estimate $I_{h3}$, we use \eqref{eq:deltaG1} together with
the bound
\begin{equation}
\label{eq:h-growth} \abs{h\bigl(u(y,v)\bigr)} \le C \bigl(1+\abs{u(y,v)}\bigr) \le C
(1+\norm{u}_v).
\end{equation}
We have
\begin{align*}
I_{h3} &\le C\int_0^t
(t-s)^{-\delta} \int_0^s (1+
\norm{u}_v) (s-v)^{\delta-\alpha-1} \int_D
\varPhi^C_{t-t^*}(x-y)\,dy\,dv\,ds
\\[-2pt]
&\le C\int_0^t (t-s)^{-\delta} \int
_0^s (1+\norm{u}_v)
(s-v)^{\delta
-\alpha-1}\,dv\,ds
\\[-2pt]
&= C\int_0^t (1+\norm{u}_v) \int
_v^t (t-s)^{-\delta} (s-v)^{\delta
-\alpha-1}
\,ds\,dv,
\end{align*}
where we choose $\delta\in(\frac{d}{d+2},1)$ so that $\delta>\alpha$.
Computing the inner integral by \eqref{eq:beta-1}, we get
\begin{equation}
\label{eq:apr-Ih3} I_{h3} \le C\int_0^t
(1+\norm{u}_v) (t-v)^{-\alpha}\,dv.
\end{equation}
Combining \eqref{eq:apr-I}, \eqref{eq:apr-I0}, \eqref{eq:apr-If}--\eqref
{eq:apr-Ih2}, \eqref{eq:apr-Ih3}, we obtain
\[
\abs{u(x,t)} \le C\xi_{\alpha,H,T} \Biggl(1+\int_0^t
\bigl(\norm{u}_s (t-s)^{-\alpha} + \norm{u}_{\alpha,1,s}
\bigr)ds \Biggr),
\]
Since $\norm{u}_s$ and $\norm{u}_{\alpha,1,s}$ are non-decreasing, the
right-hand side here is non-decreasing as well. Indeed, using the
substitution $s = zt$, the integral in the right-hand side can be
rewritten in the form
\[
t^{1-\alpha}\int_0^1\norm{u}_{zt}(1-z)^{-\alpha}dz
+ t\int_0^1\norm{u}_{\alpha,1,zt}dz.
\]
Therefore, taking suprema, we arrive at \eqref{eq:apr-term1}.
\end{proof}

\begin{lemma}
\label{l:2nd}
Under assumptions \ref{(A1)}--\ref{(A4)},
\begin{equation}
\label{eq:apr-term2} 
%
\begin{aligned}
\norm{u}_{\alpha,1,t} \le C\xi_{\alpha,H,T} \Biggl(1 +\int
_0^t \bigl(\norm{u}_s(t-s)^{-2\alpha
}
+ \norm{u}_{\alpha,1,s}(t-s)^{-\alpha} \bigr) ds \Biggr), \end{aligned}
\end{equation}
for all $t\in[0,T]$.
\end{lemma}

\begin{proof}
By \eqref{eq:mild},
\[
\abs{u(x,t)-u(x,s)} \le J_\varphi+ J_f + J_h,
\]
where
\begin{align*}
J_\varphi&= \int_D \abs{G(x,t;y,0)-G(x,s;y,0)}
\abs{\varphi(y)}\,dy,
\\*
J_f &= \abs{\int_0^t\!\!\int
_D G(x,t;y,v)f\bigl(u(y,v)\bigr)\,dy\,dv-\int
_0^s\!\! \int_D
G(x,s;y,v)f\bigl(u(y,v)\bigr)\,dy\,dv},
\\
J_h &= \sum_{j=1}^\infty
\lambda_j^{1/2} \Biggl\lvert\int_0^t
\int_D G(x,t;y,v)h\bigl(u(y,v)\bigr)e_j(y)
\,dy\,dB^H_j(v)
\\*
&\quad- \int_0^s \int_D
G(x,s;y,v)h\bigl(u(y,v)\bigr)e_j(y)\,dy\,dB^H_j(v)
\Biggr\rvert.
\end{align*}
Then
\[
\int_0^t\frac{\abs{u(x,t)-u(x,s)}}{(t-s)^{\alpha+1}} \,ds \le
K_\varphi+ K_f + K_h,
\]
where
$K_\varphi= \int_0^t\frac{J_\varphi}{(t-s)^{\alpha+1}} \,ds$,
$K_f = \int_0^t\frac{J_f}{(t-s)^{\alpha+1}} \,ds$,
$K_h = \int_0^t\frac{J_h}{(t-s)^{\alpha+1}} \,ds$.

Let $\delta\in(\frac{d}{d+2}\vee\alpha,1)$ be fixed throughout the proof.
Using the boundedness of $\varphi$ and \eqref{eq:deltaG}, we can write
\[
J_\varphi\le Cs^{-\delta}(t-s)^\delta\int
_D \varPhi^C_{v^*}(x-y)\,dy \le
Cs^{-\delta}(t-s)^\delta.
\]
Therefore,
\[
K_\varphi\le C \int_0^ts^{-\delta}(t-s)^{\delta-\alpha-1}
\,ds \le C.
\]

Let us consider $K_f$.
We have that
\begin{align*}
J_f &\le\int_s^t\!\!\int
_D \abs{G(x,t;y,v)f\bigl(u(y,v)\bigr)}\,dy\,dv
\\[-2pt]
&\quad+\int_0^s\!\!\int_D
\abs{G(x,t;y,v) - G(x,s;y,v)} \abs {f\bigl(u(y,v)\bigr)}\,dy\,dv.
\end{align*}
Consequently,
\[
K_f \le K_{f}' + K_{f}'',
\]
where
\begin{align*}
K_{f}' &= \int_0^t
(t-s)^{-\alpha-1} \int_s^t\!\!\int
_D \abs {G(x,t;y,v)f\bigl(u(y,v)\bigr)}\,dy\,dv\,ds,
\\[-2pt]
K_{f}'' &= \int_0^t
(t-s)^{-\alpha-1}\int_0^s\!\!\int
_D \abs{G(x,t;y,v) - G(x,s;y,v)} \abs{f\bigl(u(y,v)\bigr)}
\,dy\,dv\,ds.
\end{align*}
By \eqref{eq:f-growth} and the Gaussian\index{Gaussian property} property of $G$,
\begin{align}
K_{f}' &\le C\int_0^t
(t-s)^{-\alpha-1} \int_s^t(1 +
\norm{u}_v)\,dv\,ds \notag
\\[-2pt]
&= C\int_0^t (1 + \norm{u}_v) \int
_0^v (t-s)^{-\alpha-1}\,ds\,dv \notag
\\[-2pt]
&\le C\int_0^t (1 + \norm{u}_v)
(t-v)^{-\alpha}\,dv\le C \int_0^t (1 +
\norm{u}_v) (t-v)^{-2\alpha}\,dv. \label{eq:Kf'}
\end{align}
%
In order to estimate $K_{f}''$, we apply \eqref{eq:f-growth} and \eqref
{eq:deltaG}, change the order of integration and then use \eqref{eq:beta-1}:
\begin{align}
K_{f}'' &\le C \int_0^t
(t-s)^{\delta-\alpha-1}\int_0^s
(s-v)^{-\delta} (1 + \norm{u}_v) \int_D
\varPhi^C_{v^*-v}(x-y)\,dy\,dv\,ds \notag
\\
&\le C \int_0^t (t-s)^{\delta-\alpha-1}\int
_0^s (s-v)^{-\delta} (1 +
\norm{u}_v)\,dv\,ds \notag
\\
&= C \int_0^t (1 + \norm{u}_v)
\int_v^t(t-s)^{\delta-\alpha-1} (s-v)^{-\delta}
\,ds\,dv \notag
\\
&\le C \int_0^t (1 + \norm{u}_v)
(t-v)^{-\alpha}\,dv\le C \int_0^t (1 +
\norm{u}_v) (t-v)^{-2\alpha}\,dv. \label{eq:Kf''}
\end{align}

Next, consider $K_h$.
We have
\begin{align*}
J_h &\le\sum_{j=1}^\infty
\lambda_j^{1/2} \Biggl\lvert\int_s^t
\!\! \int_D G(x,t;y,v)h\bigl(u(y,v)\bigr)e_j(y)
\,dy\, dB^H_j(v) \Biggr\rvert
\\
&\quad+\sum_{j=1}^\infty\lambda_j^{1/2}
\Biggl\lvert\int_0^s \!\! \int
_D \bigl(G(x,t;y,v) - G(x,s;y,v)\bigr) h\bigl(u(y,v)\bigr)
e_j(y)\,dy\,dB^H_j(v) \Biggr\rvert
\\
&\eqqcolon J_h'+J_h''.
\end{align*}
The integrals with respect to fractional Brownian motions\index{fractional Brownian motions} can be
bounded similarly to $I_h$, applying \eqref{eq:bound-int} and \ref{(A4)}.
We obtain
\begin{align}
J_h' &\le C\xi_{\alpha,H,t} \int
_s^t \Biggl( \frac{\int_D \abs
{G(x,t;y,v)h(u(y,v))} dy}{(v-s)^{\alpha}} \notag
\\*
&\quad+ \int_s^v \frac{\int_D \abs{G(x,t;y,v)h(u(y,v)) -
G(x,t;y,r)h(u(y,r))} dy}{(v-r)^{\alpha+1}} \,dr
\Biggr)dv \notag
\\
&\le C \xi_{\alpha,H,T} (J_{h1}+J_{h2}+J_{h3}),
\label{eq:Jh'}
\end{align}
where
\begin{align*}
J_{h1} &= \int_s^t
(v-s)^{-\alpha}\! \int_D \abs{G(x,t;y,v)h\bigl(u(y,v)
\bigr)} dy\,dv,
\\
J_{h2} &= \int_s^t\!\!\int
_s^v (v-r)^{-\alpha-1} \int
_D \abs{G(x,t;y,v)} \abs{h\bigl(u(y,v)\bigr) - h
\bigl(u(y,r)\bigr)} dy \,dr\,dv,
\\
J_{h3} &= \int_s^t\!\!\int
_s^v (v-r)^{-\alpha-1} \int
_D \abs{G(x,t;y,v) - G(x,t;y,r)} \abs{h\bigl(u(y,r)\bigr)}
dy \,dr\,dv.
\end{align*}

Similarly,
\begin{align*}
J_h'' &\le C\xi_{\alpha,H,s} \int
_0^s \Biggl( v^{-\alpha} \int
_D \abs{G(x,t;y,v) - G(x,s;y,v)} \abs{h\bigl(u(y,v)\bigr)}
dy
\\*
&\quad{}+ \int_0^v (v-r)^{-\alpha-1}\int
_D \bigl\lvert \bigl(G(x,t;y,v) - G(x,s;y,v) \bigr) h
\bigl(u(y,v)\bigr)
\\*
&\quad{}- \bigl(G(x,t;y,r) - G(x,s;y,r) \bigr) h\bigl(u(y,r)\bigr) \bigr
\rvert dy \,dr \Biggr)dv
\\
&\le C \xi_{\alpha,H,T} (J_{h4}+J_{h4}+J_{h6}),
\end{align*}
where
\begin{align*}
J_{h4} &= \int_0^s v^{-\alpha}
\int_D \abs{G(x,t;y,v) - G(x,s;y,v)} \abs{h\bigl(u(y,v)
\bigr)} dy\,dv,
\\
J_{h5} &= \int_0^s\!\!\int
_0^v (v-r)^{-\alpha-1}\int
_D \abs{G(x,t;y,v) - G(x,s;y,v)}
\\*
&\quad\times\abs{h\bigl(u(y,v)\bigr) - h\bigl(u(y,r)\bigr)} dy \,dr \,dv,
\\
J_{h6} &= \int_0^s\!\!\int
_0^v (v-r)^{-\alpha-1} \int
_D \bigl\lvert G(x,t;y,v) - G(x,s;y,v)
\\*
&\quad- G(x,t;y,r) + G(x,s;y,r)\bigr\rvert\abs{h\bigl(u(y,r)\bigr)} dy \,dr
\,dv.
\end{align*}

Now it remains to estimate
$K_{hi}=\int_0^t (t-s)^{-\alpha-1} J_{hi} \,ds$, $i=1,2,\dots,6$.

In order to bound $K_{h1}$ we apply successively \eqref{eq:h-growth},
\eqref{eq:gaus} and \eqref{eq:beta-2} (with $a=2\alpha$, $b=1-\alpha$):
\begin{align}
K_{h1} &\le C\int_0^t
(t-s)^{-\alpha-1} \int_s^t
(v-s)^{-\alpha} (1+\norm{u}_v) \int_D
\abs{G(x,t;y,v)} dy\,dv\,ds \notag
\\
&\le C\int_0^t (t-s)^{-\alpha-1} \int
_s^t (v-s)^{-\alpha} (1+\norm
{u}_v)\,dv\,ds \notag
\\
&= C\int_0^t (1+\norm{u}_v) \int
_0^v (t-s)^{-\alpha-1} (v-s)^{-\alpha}
\,ds \,dv \notag
\\
&\le C\int_0^t (1+\norm{u}_v)
(t-v)^{-2\alpha} \,dv. \label{eq:Kh1}
\end{align}
By the Lipschitz continuity\index{Lipschitz continuity} of $h$,
\[
K_{h2} \le C\int_0^t
(t-s)^{-\alpha-1} \int_s^t\!\! \int
_D \abs{G(x,t;y,v)} \int_s^v
\frac{\abs{u(y,v) -
u(y,r)}}{(v-r)^{\alpha+1}} \,dr \,dy \,dv\,ds.
\]
According to the definition, the inner integral can be bounded by $\norm
{u}_{\alpha,1,v}$. Then we use the Gaussian\index{Gaussian property} property of $G$ to obtain
\begin{align*}
K_{h2} &\le C\int_0^t
(t-s)^{-\alpha-1} \int_s^t
\norm{u}_{\alpha,1,v} \int_D \abs{G(x,t;y,v)} \,dy \,dv
\,ds
\\
&\le C\int_0^t (t-s)^{-\alpha-1} \int
_s^t \norm{u}_{\alpha,1,v} \,dv\,ds
\\
&= C\int_0^t \norm{u}_{\alpha,1,v} \int
_0^v (t-s)^{-\alpha-1} \,ds\,dv
\\
&\le C\int_0^t \norm{u}_{\alpha,1,v}
(t-v)^{-\alpha}\,dv.
\end{align*}
In order to estimate $K_{h3}$, we use \eqref{eq:h-growth} and \eqref
{eq:deltaG1}, and then \eqref{eq:beta-1}:
\begin{align}
K_{h3} &\le C\int_0^t
(t-s)^{-\alpha-1} \int_s^t
(t-v)^{-\delta} \int_s^v
(v-r)^{\delta-\alpha-1} (1+\norm{u}_r) \notag
\\*
&\quad\times\int_D \varPhi^C_{t-t^*}(x-y)
dy \,dr\,dv\,ds \notag
\\
&\le C\int_0^t (t-s)^{-\alpha-1} \int
_s^t (t-v)^{-\delta} \int
_s^v (v-r)^{\delta-\alpha-1} (1+
\norm{u}_r) \,dr\,dv\,ds \notag
\\
&= C\int_0^t (t-s)^{-\alpha-1} \int
_s^t (1+\norm{u}_r) \int
_r^t (t-v)^{-\delta} (v-r)^{\delta-\alpha-1}
\,dv\,dr\,ds \notag
\\
&\le C\int_0^t (t-s)^{-\alpha-1} \int
_s^t (1+\norm{u}_r)
(t-r)^{-\alpha} \,dr\,ds \notag
\\
&= C\int_0^t (1+\norm{u}_r)
(t-r)^{-\alpha} \int_0^r
(t-s)^{-\alpha-1} \,ds\,dr \notag
\\
&\le C\int_0^t (1+\norm{u}_r)
(t-r)^{-2\alpha} \,dr. \label{eq:Kh3}
\end{align}

The term $K_{h4}$ can be bounded similarly with the help of \eqref
{eq:h-growth}, \eqref{eq:deltaG} and \eqref{eq:beta-1}:
\begin{align*}
K_{h4} &\le C\int_0^t
(t-s)^{\delta-\alpha-1} \int_0^s (1+
\norm{u}_v) v^{-\alpha} (s-v)^{-\delta} \int
_D \varPhi^C_{v^*-v}(x-y) dy\,dv\,ds
\\
&\le C\int_0^t (t-s)^{\delta-\alpha-1} \int
_0^s (1+\norm{u}_v) v^{-\alpha}
(s-v)^{-\delta}\,dv\,ds 
\\
&= C\int_0^t (1+\norm{u}_v)
v^{-\alpha} \int_v^t (t-s)^{\delta-\alpha-1}
(s-v)^{-\delta}\,ds\,dv 
\\
&\le C\int_0^t (1+\norm{u}_v)
v^{-\alpha} (t-v)^{-\alpha}\,dv. 
\end{align*}
Since $(1+\norm{u}_v) (t-v)^{-\alpha}$ is non-decreasing and $v^{-\alpha
}$ is non-increasing, using the rearrangement inequality, we obtain
\begin{equation}
\label{eq:Kh4} K_{h4}\le C\int_0^t
(1+\norm{u}_v) (t-v)^{-2\alpha}\,dv.
\end{equation}
From the Lipschitz continuity\index{Lipschitz continuity} of $h$ we get
\begin{align*}
K_{h5} &\le C\int_0^t
(t-s)^{-\alpha-1} \int_0^s\!\!\int
_D \abs{G(x,t;y,v) - G(x,s;y,v)}
\\*
&\quad\times \int_0^v \frac{\abs{u(y,v) - u(y,r)}}{(v-r)^{\alpha+1}} \,dr
\,dy \,dv\,ds.
\end{align*}
From \eqref{eq:deltaG}, \eqref{eq:beta-1} it follows that
\begin{align*}
K_{h5} &\le C\int_0^t
(t-s)^{\delta-\alpha-1} \int_0^s
\norm{u}_{\alpha,1,v} (s-v)^{-\delta} \int_D
\varPhi^C_{v^*-v}(x-y)\,dy \,dv\,ds
\\
&\le C\int_0^t (t-s)^{\delta-\alpha-1} \int
_0^s \norm{u}_{\alpha,1,v}
(s-v)^{-\delta} \,dv\,ds
\\
&= C\int_0^t \norm{u}_{\alpha,1,v} \int
_v^t(t-s)^{\delta-\alpha-1} (s-v)^{-\delta}
\,ds\,dv
\\
&\le C\int_0^t \norm{u}_{\alpha,1,v}
(t-v)^{-\alpha}\,dv.
\end{align*}
Finally, we estimate $K_{h6}$, using \eqref{eq:deltaG2}, \eqref
{eq:beta-2}, \eqref{eq:apr-I0} and \eqref{eq:h-growth}:
\begin{align}
K_{h6} &\le C\int_0^t
(t-s)^{-\alpha-1} \int_0^s \int
_0^v (1+\norm {u}_r)
(v-r)^{-\alpha-1} \notag
\\[-2pt]
&\quad{}\times \int_r^v \int
_s^t (\theta-\tau)^{-2} \int
_D \varPhi^C_{\theta-\tau}(x-y) dy\,d\theta
\,d\tau\,dy \,dr\,dv\,ds\notag
\\[-2pt]
& \le C\int_0^t (1+\norm{u}_\theta)
\int_0^\theta(t-s)^{-\alpha-1} \notag
\\[-2pt]
&\quad{} \times\int_0^s \int_0^v(
\theta-\tau)^{-2}\int_0^\tau
(v-r)^{-\alpha
-1}dr\,d\tau\,dv\,ds\,d\theta\notag
\\[-2pt]
& \le C \int_0^t (1+\norm{u}_\theta)
\int_0^\theta(t-s)^{-\alpha-1} \int
_0^s \int_0^v(
\theta-\tau)^{-2} (v-\tau)^{-\alpha} d\tau\,dv\,ds\, d\theta\notag
\\[-2pt]
& \le C \int_0^t (1+\norm{u}_\theta)
\int_0^\theta(t-s)^{-\alpha-1} \int
_0^s (\theta-v)^{-\alpha-1}dv\,ds\,d\theta
\notag\notag
\\[-2pt]
&\le C \int_0^t (1+\norm{u}_\theta)
\int_0^\theta(t-s)^{-\alpha-1} (
\theta-s)^{-\alpha}ds\,d\theta\notag
\\[-2pt]
&\le C \int_0^t (1+\norm{u}_\theta)
(t-\theta)^{-2\alpha} d\theta. \label{eq:Kh6}
\end{align}

Combining the obtained bounds for $K_\varphi$, $K_{f}'$, $K_{f}''$ and
$K_{hi}$, $i=1,\dots,6$, we obtain
\begin{align*}
&\int
_0^t\frac{\abs{u(x,t)-u(x,v)}}{(t-v)^{\alpha+1}} \,dv \le C
\xi_{\alpha,H,T}
\\[-2pt]
&\quad{}\times \Biggl(1 +\int_0^t \bigl(
\norm{u}_v(t-v)^{-2\alpha
} + \norm{u}_{\alpha,1,v}(t-v)^{-\alpha}
\bigr) dv \Biggr).
\end{align*}
Since the integral in the right-hand side can be rewritten in the form
\[
t^{1-2\alpha}\int_0^1\norm{u}_{zt}(1-z)^{-2\alpha}dz
+ t^{1-\alpha}\int_0^1\norm{u}_{\alpha,1,zt}(1-z)^{-\alpha}dz,
\]
it is a non-decreasing function.
Therefore, taking suprema, we arrive at \eqref{eq:apr-term2}.
%
\end{proof}

\begin{proof}[Proof of Proposition~\ref{pr:1}]
Lemmata \ref{l:1st} and \ref{l:2nd} allow us to use a kind of
two-dimensional Gr\"onwall argument, proposed in \cite[Lemma
4.1]{jumps}. Namely, for some $\lambda>0$, which will be chosen later, define
\[
f_1(\lambda) = \sup_{t\in[0,T]} e^{-\lambda t}
\norm{u}_t, \quad f_2(\lambda) = \sup_{t\in[0,T]}
e^{-\lambda t}\norm{u}_{\alpha,1,t} %
\]
and denote for shortness $\xi= \xi_{\alpha,H,T}$.
From \eqref{eq:apr-term1} we get
\begin{align}
f_1(\lambda)&\le C\xi \Biggl(1+ \sup_{t\in[0,T]}
e^{-\lambda t} \int_0^t \bigl(e^{\lambda s}
f_1(\lambda) (t-s)^{-\alpha} + e^{\lambda
s}f_2(
\lambda) \bigr)ds \Biggr)\notag
\\[-2pt]
& \le C\xi \Biggl(1 + f_1(\lambda)\sup_{t\in[0,T]} \int
_0^t e^{-\lambda
(t-s)} (t-s)^{-\alpha}ds
\notag
\\
& \quad{} + f_2(\lambda)\sup
_{t\in[0,T]} \int_0^t
e^{-\lambda(t-s)}ds \Biggr)\notag
\\
& \le C\xi \Biggl(1 + f_1(\lambda) \int_0^\infty
e^{-\lambda u} u^{-\alpha
}du + f_2(\lambda)\int
_0^\infty e^{-\lambda u}du \Biggr)\notag
\\
& \le C\xi \bigl(1 + \lambda^{\alpha-1} f_1(\lambda) + \lambda
^{-1}f_2(\lambda) \bigr). \label{eq:gron1}
\end{align}
Similarly, from \eqref{eq:apr-term2} we get
\begin{align}
f_2(\lambda)\le C\xi \bigl(1+f_1(\lambda)
\lambda^{2\alpha-1} + f_2(\lambda )\lambda^{\alpha-1} \bigr) .
\label{eq:gron2}
\end{align}
Let $K$ be the largest of the constants in \eqref{eq:gron1} and \eqref
{eq:gron2}; without loss of generality we can assume that $K\ge1$.
Setting $\lambda= (4K\xi)^{1/(1-\alpha)}$ and plugging it into \eqref
{eq:gron2}, we obtain
\begin{align*}
f_2(\lambda)& \le K\xi \bigl(1+\lambda^{2\alpha-1}
f_1(\lambda) + \lambda^{-1}f_2(\lambda) \bigr)
\\
& = \frac{1}4 \lambda^{1-\alpha} \bigl(1+ \lambda^{2\alpha-1}
f_1(\lambda )+ f_2(\lambda) \bigr)
\\
& = \frac{1}4 \bigl( \lambda^{1-\alpha} + \lambda^\alpha
f_1(\lambda) + f_2(\lambda) \bigr),
\end{align*}
whence $f_2(\lambda) \le\frac{1}3  (\lambda^{1-\alpha} + \lambda
^\alpha f_1(\lambda) )$. Plugging this into \eqref{eq:gron1} and
noting that $\lambda\ge1$, we get
\begin{align*}
f_1(\lambda) &\le\frac{1}4 \lambda^{1 - \alpha} \biggl(1+
\lambda^{\alpha
-1}f_1(\lambda) + \frac{1}3
\lambda^{-\alpha} + \frac{1}3 \lambda^{\alpha
-1}f_1(
\lambda) \biggr)
\\
& \le\lambda^{1-\alpha} + \frac{1}{3}f_1(\lambda),
\end{align*}
whence
$f_1(\lambda) \le\frac{3}2 \lambda^{1-\alpha} = 6K\xi$. It follows that
\[
\norm{u}_T \le f_1(\lambda)e^{\lambda T} \le6K\xi\exp
\set{(4K\xi )^{1/(1-\alpha)}T}\le C \exp\set{C\xi^{1/(1-\alpha)}}. %
\]
Similarly,
\[
\norm{u}_{\alpha,1,T}\le f_2(\lambda)e^{\lambda T}\le
\lambda^{\alpha
}f_1(\lambda)e^{\lambda T}\le C \exp\set{C
\xi^{1/(1-\alpha)}}. %
\]
The statement then follows from adding these estimates.
\end{proof}
By Fernique's theorem, $\E\exp\{a \xi^2\}<\infty$ for some $a>0$.
Since $\frac{1}{1-\alpha}<2$, we have $C\xi^{1/(1-\alpha)}< a \xi^2 +
C$, so Proposition~\ref{pr:1} implies existence of moments of the solution.

\begin{corollary}\label{cor:1}
Under assumptions \ref{(A1)}--\ref{(A4)}, the solution $u$ to \eqref
{eq:spde} satisfies
\begin{equation*}
\E\norm{u}_{\alpha,\infty,T}^p <\infty
\end{equation*}
for all $p>0$. In particular,
\begin{equation*}
\E\sup_{t\in[0,T],x\in D} \abs{u(t,x)}^p <\infty
\end{equation*}
for all $p>0$.
\end{corollary}

\section{Existence and uniqueness of mild solution\index{mild solution}}
\label{sec:mild}

The following theorem is the main result of the article.

\begin{theorem}\label{th:mild}
Let $H\in (1/2,1 )$, $\alpha\in(1-H,1/2)$.
Assume that Hypotheses \ref{(A1)}--\ref{(A4)} hold.
Then the problem \eqref{eq:spde} has a unique mild solution.
\end{theorem}

Under the standing assumptions, the existence of a mild solution\index{mild solution} was
established in \cite[Th.~2.3(a)]{SSV09}. Hence, it remains to prove the
uniqueness.

Let $u$ and $\tilde u$ be two mild solutions\index{mild solution} to the problem \eqref{eq:spde}.
In order to prove that $u$ and $\tilde u$ coincide
we shall establish that the norm
\begin{equation}
\label{eq:norm-du} \norm{u - \tilde u}_{\alpha,\infty,T} =\norm{u - \tilde
u}_{T} + \norm{u - \tilde u}_{\alpha,1,T}
\end{equation}
is equal to zero.
The proof of this fact is carried out similarly to that of
Proposition~\ref{pr:1}, using the bounds\index{bounds}
\begin{align}
&\abs{f \bigl(u(y,s) \bigr) - f \bigl(\tilde u(y,s) \bigr)} +
\abs{h \bigl(u(y,s) \bigr) - h \bigl(\tilde u(y,s) \bigr)}
\notag\\*
&\quad{}\le C \abs{u(y,s) - \tilde u(y,s)} \le C \norm{u - \tilde u}_s
\label{eq:fh-lipsh}
\end{align}
instead of \eqref{eq:f-growth} and \eqref{eq:h-growth}.
Therefore we omit some details.
As above, we first obtain the upper bounds for each of two terms in the
right-hand side of the norm \eqref{eq:norm-du}.

Let $\xi=\xi_{\alpha,H,T}$.
Denote also
\[
\eta= 1 + \norm{u}_{\alpha,1,T} + \norm{\tilde u}_{\alpha,1,T}.
\]
Corollary~\ref{cor:1} implies that $\eta$ is finite a.\,s.
\begin{lemma}
\label{l:uniq-1st}
Under assumptions \ref{(A1)}--\ref{(A4)},
\[
\norm{u - \tilde u}_t \le C \xi\eta\int_0^t
\norm{u - \tilde u}_{\alpha,\infty,s} (t-s)^{-\alpha}ds
\]
for all $t\in[0,T]$.
\end{lemma}

\begin{proof}
By \eqref{eq:mild},
\begin{equation}
\label{eq:du} \abs{u(x,t) - \tilde u(x,t)} \le P_f +
P_h,
\end{equation}
where
\begin{align*}
P_f &= \int_0^t\int
_D \abs{G(x,t;y,s)} \abs{f \bigl(u(y,s) \bigr) - f \bigl(
\tilde u (y,s) \bigr)} dy\,ds,
\\
P_h &= \sum_{j=1}^\infty
\lambda_j^{1/2} \abs{\int_0^t
\!\! \int_D G(x,t;y,s) \bigl(h \bigl(u(y,s) \bigr) - h
\bigl(\tilde u(y,s) \bigr) \bigr) e_j(y)\,dy\,dB^H_j(s)}
\!.
\end{align*}
Using the bound
\eqref{eq:fh-lipsh}
and the Gaussian\index{Gaussian property} property of $G$ we immediately get
\begin{equation}
\label{eq:Pf} P_f \le C \int_0^t
\norm{u - \tilde u}_s ds.
\end{equation}

Further, similarly to \eqref{eq:apr-Ih},
\[
P_h \le C \xi(P_{h1}+P_{h2}+P_{h3}),
\]
where
\begin{align}
P_{h1} &= \int_0^t s^{-\alpha}
\int_D \abs{G(x,t;y,s)} \abs{h \bigl(u(y,s) \bigr) - h
\bigl(\tilde u(y,s) \bigr)}\,dy\,ds, \notag
\\
P_{h2} &= \int_0^t \int
_0^s (s-v)^{-\alpha-1}\int
_D \abs{G(x,t;y,s)} \notag
\\*
&\quad\times\abs{h \bigl(u(y,s) \bigr) - h \bigl(\tilde u(y,s) \bigr) - h
\bigl(u(y,v) \bigr) + h \bigl(\tilde u(y,v) \bigr)} dy\,dv\,ds, \label{eq:Ph2-def}
\\
P_{h3} &= \int_0^t \int
_0^s (s-v)^{-\alpha-1}\int
_D \abs{G(x,t;y,s) - G(x,t;y,v)} \notag
\\*
&\quad\times\abs{h \bigl(u(y,v) \bigr) - h \bigl(\tilde u(y,v) \bigr)} dy\,dv
\,ds. \notag
\end{align}
The bounds\index{bounds}
\begin{align}
P_{h1} &\le C \int_0^t
(t-s)^{-\alpha} \norm{u - \tilde u}_s ds \label{eq:Ph1}
\shortintertext{and} P_{h3} &\le C \int_0^t
(t-s)^{-\alpha} \norm{u - \tilde u}_s ds \label{eq:Ph3}
\end{align}
are obtained analogously to the bounds\index{bounds} \eqref{eq:apr-Ih1} and \eqref
{eq:apr-Ih3}.

According to \cite[Lemma 7.1]{NR02}, the assumption \ref{(A3)} implies
that for any $x_1$, $x_2$, $x_3$, $x_4$,
\begin{align*}
\abs{h(x_1)-h(x_2)-h(x_3)+h(x_4)}
&{}\le C\abs{x_1-x_2-x_3+x_4}
\\
&\quad{} + C \abs{x_1-x_3} (\abs{x_1-x_2}+
\abs{x_3-x_4} ).
\end{align*}
Therefore, we can write
\begin{align}
\MoveEqLeft \int_0^s \frac{\abs{h (u(y,s) ) - h (\tilde
u(y,s) )
- h (u(y,v) ) + h (\tilde u(y,v) )}}{(s-v)^{\alpha
+1}}\,dv
\notag
\\*
&\le C\int_0^s \frac{\abs{u(y,s)-\tilde u(y,s)-u(y,v)
+ \tilde u(y,v)}}{(s-v)^{\alpha+1}}\,dv \notag
\\
&\quad+ C \abs{u(y,s)-\tilde u(y,s)} \int_0^s
\frac{\abs{u(y,s)-u(y,v)}
+ \abs{\tilde u(y,s) - \tilde u(y,v)}}{(s-v)^{\alpha+1}}\,dv \notag
\\
&\le C \norm{u - \tilde u}_{\alpha,1,s} + C \norm{u - \tilde u}_s
(\norm{u}_{\alpha,1,s}+\norm{\tilde u}_{\alpha,1,s} ) \notag
\\
&\le C \eta\norm{u - \tilde u}_{\alpha,\infty,s}. \label{eq:h4dif}
\end{align}
Inserting the bound \eqref{eq:h4dif} into \eqref{eq:Ph2-def}, we get
\begin{equation}
\label{eq:Ph2} P_{h2} \le C \eta \int_0^t
\norm{u - \tilde u}_{\alpha,\infty,s} \int_D
\abs{G(x,t;y,s)} dy\,ds \le C \eta\int_0^t
\norm{u - \tilde u}_{\alpha,\infty,s} ds.
\end{equation}

Combining \eqref{eq:du}, \eqref{eq:Pf}, \eqref{eq:Ph1}, \eqref{eq:Ph3},
and \eqref{eq:Ph2} we arrive at
\[
\abs{u(x,t) - \tilde u(x,t)} \le C \xi\eta\int_0^t
\norm{u - \tilde u}_{\alpha,\infty,s} (t-s)^{-\alpha}ds.
\]
We conclude the proof similarly to that of Lemma~\ref{l:1st}, t using
the monotonicity of the right-hand side.
\end{proof}

\begin{lemma}
\label{l:uniq-2nd}
Under assumptions \ref{(A1)}--\ref{(A4)},
\begin{equation}
\label{eq:uniq-term2} \norm{u - \tilde u}_{\alpha,1,t} \le C \xi\eta\int
_0^t\norm{u - \tilde u}_{\alpha,\infty,s}
(t-s)^{-2\alpha} ds
\end{equation}
for all $t\in[0,T]$.
\end{lemma}

\begin{proof}
As in the proof of Lemma \ref{l:2nd}, we can write
\begin{align*}
\MoveEqLeft[.4] \int_0^t \frac{\abs{u(x,t) - \tilde u(x,t) - u(x,s)
+ \tilde u(x,s)}}{(t-s)^{\alpha+1}}
\,ds
\\[-2pt]
&\le \int_0^t (t-s)^{-\alpha-1} \Biggl( \Biggl
\lvert\int_0^t\!\!\int_D
G(x,t;y,v) \bigl( f \bigl(u(y,v) \bigr) - f \bigl(\tilde u(y,v) \bigr) \bigr)
\,dy\,dv
\\[-2pt]
&\quad-\int_0^s\!\!\int_D
G(x,s;y,v) \bigl( f \bigl(u(y,v) \bigr) - f \bigl(\tilde u(y,v) \bigr) \bigr)
\,dy\,dv \Biggr\rvert
\\[-2pt]
&\quad+ \sum_{j=1}^\infty
\lambda_j^{1/2} \Biggl\lvert\int_0^t
\!\! \int_D G(x,t;y,v) \bigl( h \bigl(u(y,v) \bigr) - h
\bigl(\tilde u(y,v) \bigr) \bigr)e_j(y)\,dy\,dB^H_j(v)
\\[-2pt]
&\quad- \int_0^s \!\!\int
_D G(x,s;y,v) \bigl( h \bigl(u(y,v) \bigr) - h \bigl(\tilde
u(y,v) \bigr) \bigr)e_j(y)\,dy\,dB^H_j(v)
\Biggr\rvert \Biggr) ds
\\[-2pt]
&\le Q_f'+Q_f''+Q_h'+Q_h'',
\end{align*}
where
\begin{align*}
Q_f' &= \int_0^t
(t-s)^{-\alpha-1} \int_s^t\!\!\int
_D \abs{G(x,t;y,v)} \abs{ f \bigl(u(y,v) \bigr) - f \bigl(
\tilde u(y,v) \bigr) } \,dy\,dv\,ds,
\\[-2pt]
Q_f'' &= \int_0^t
(t-s)^{-\alpha-1} \int_0^s\!\!\int
_D \abs{G(x,t;y,v) - G(x,s;y,v)}
\\[-2pt]
&\quad\times\abs{ f \bigl(u(y,v) \bigr) - f \bigl(\tilde u(y,v) \bigr) } \,dy
\,dv\,ds,
\\[-2pt]
Q_h' &= \sum_{j=1}^\infty
\lambda_j^{1/2} \int_0^t
(t-s)^{-\alpha-1}
\\[-2pt]
&\quad\times\abs{\int_s^t\!\! \int
_D G(x,t;y,v) \bigl( h \bigl(u(y,v) \bigr) - h \bigl(\tilde
u(y,v) \bigr) \bigr)e_j(y)\,dy\, dB^H_j(v)}
ds,
\\[-2pt]
Q_h'' &= \sum
_{j=1}^\infty\lambda_j^{1/2} \int
_0^t (t-s)^{-\alpha-1} \Biggl\lvert\int
_s^t\!\! \int_D
\bigl(G(x,t;y,v) - G(x,s;y,v) \bigr)
\\[-2pt]
&\quad\times \bigl( h \bigl(u(y,v) \bigr) - h \bigl(\tilde u(y,v) \bigr)
\bigr)e_j(y)\,dy\,dB^H_j(v) \Biggr\rvert
\,ds.
\end{align*}

Similarly to \eqref{eq:Kf'}, \eqref{eq:Kf''} and \eqref{eq:Jh'}, we get
\begin{align*}
Q_{f}' &\le C\int_0^t
\norm{u-\tilde u}_v (t-v)^{-2\alpha}\,dv,
\\
Q_{f}'' &\le C\int_0^t
\norm{u-\tilde u}_v (t-v)^{-2\alpha}\,dv, \shortintertext{and}
Q_h' &\le C \xi(Q_{h1}+Q_{h2}+Q_{h3}),
\end{align*}
where
\begin{align*}
Q_{h1} &= \int_0^t
(t-s)^{-\alpha-1} \int_s^t
(v-s)^{-\alpha}
\\*
&\quad\times\int_D \abs{G(x,t;y,v)}\abs{h \bigl(u(y,v)
\bigr) - h \bigl(\tilde u(y,v) \bigr)} dy\,dv\,ds,
\\
Q_{h2} &= \int_0^t
(t-s)^{-\alpha-1} \int_s^t\!\!\int
_s^v (v-r)^{-\alpha-1} \int
_D \abs{G(x,t;y,v)}
\\*
&\quad\times \abs{h \bigl(u(y,v) \bigr) - h \bigl(\tilde u(y,v) \bigr) - h
\bigl(u(y,r) \bigr) + h \bigl(\tilde u(y,r) \bigr)} dy \,dr\,dv\,ds,
\\
Q_{h3} &= \int_0^t
(t-s)^{-\alpha-1} \int_s^t\!\!\int
_s^v (v-r)^{-\alpha-1}
\\*
&\quad\times \int_D \abs{G(x,t;y,v) - G(x,t;y,r)} \abs{h
\bigl(u(y,r) \bigr) - h \bigl(\tilde u(y,r) \bigr)} dy \,dr\,dv\,ds.
\end{align*}
Further, we estimate
\begin{align*}
Q_{h1} &\le C\int_0^t \norm{u-
\tilde u}_v (t-v)^{-2\alpha}\,dv,
\\
Q_{h3} &\le C\int_0^t \norm{u-
\tilde u}_v (t-v)^{-2\alpha}\,dv,
\end{align*}
similarly to \eqref{eq:Kh1} and \eqref{eq:Kh3}.
Then applying \eqref{eq:h4dif}, the Gaussian\index{Gaussian property} property of $G$, and
integrating with respect to $s$, we obtain
\begin{align*}
Q_{h2} &\le C \eta\int_0^t
(t-s)^{-\alpha-1} \int_s^t \norm{u-\tilde
u}_{\alpha,\infty,v} \int_D \abs{G(x,t;y,v)} dy\,dv\,ds,
\\
&\le C \eta\int_0^t (t-v)^{-\alpha}
\norm{u-\tilde u}_{\alpha,\infty,v}\,dv \le C \eta\int_0^t
(t-v)^{-2\alpha} \norm{u-\tilde u}_{\alpha,\infty,v}\,dv.
\end{align*}
Finally,
\[
Q_h'' \le C \xi(Q_{h4}+Q_{h5}+Q_{h6}),
\]
where
\begin{align*}
Q_{h4} &= \int_0^t
(t-s)^{-\alpha-1} \int_0^s v^{-\alpha}
\int_D \abs{G(x,t;y,v) - G(x,s;y,v)}
\\*
&\quad\times \abs{h \bigl(u(y,v) \bigr) - h \bigl(\tilde u(y,v) \bigr)} dy\,dv
\,ds,
\\
Q_{h5} &= \int_0^t
(t-s)^{-\alpha-1} \int_0^s\!\!\int
_0^v (v-r)^{-\alpha-1}\int
_D \abs{G(x,t;y,v) - G(x,s;y,v)}
\\*
&\quad\times\abs{h \bigl(u(y,v) \bigr) - h \bigl(\tilde u(y,v) \bigr) -h
\bigl(u(y,r) \bigr) + h \bigl(\tilde u(y,r) \bigr)} dy \,dr \,dv\,ds,
\\
Q_{h6} &= \int_0^t
(t-s)^{-\alpha-1} \int_0^s\!\!\int
_0^v (v-r)^{-\alpha-1} \int
_D \bigl\lvert G(x,t;y,v) - G(x,s;y,v)
\\*
&\quad- G(x,t;y,r) + G(x,s;y,r)\bigr\rvert \abs{h \bigl(u(y,r) \bigr) - h
\bigl(\tilde u(y,r) \bigr)} dy \,dr\,dv\,ds.
\end{align*}
Arguing as in the proof of Lemma \ref{l:2nd}, we can prove the bounds\index{bounds}
\begin{align*}
Q_{h4} &\le C\int_0^t \norm{u-
\tilde u}_v (t-v)^{-2\alpha}\,dv,
\\
Q_{h6} &\le C\int_0^t \norm{u-
\tilde u}_v (t-v)^{-2\alpha}\,dv,
\end{align*}
which are similar to \eqref{eq:Kh4} and \eqref{eq:Kh6}.
Finally, using \eqref{eq:h4dif}, \eqref{eq:deltaG}, and \eqref
{eq:beta-1}, we get
\begin{align*}
Q_{h5} &\le C \eta\int_0^t
(t-s)^{-\alpha-1} \int_0^s \norm{u-\tilde
u}_{\alpha,\infty,v}
\\*
&\quad\times \int_D \abs{G(x,t;y,v) - G(x,s;y,v)}\,dy
\,dv\,ds
\\
&\le C \eta\int_0^t (t-s)^{\delta-\alpha-1} \int
_0^s \norm{u-\tilde u}_{\alpha,\infty,v}
(s-v)^{-\delta}\,dv\,ds
\\
&\le C \eta\int_0^t \norm{u-\tilde
u}_{\alpha,\infty,v} (t-v)^{-\alpha}\,dv
\\
&\le C \eta\int_0^t \norm{u-\tilde
u}_{\alpha,\infty,v} (t-v)^{-2\alpha}\,dv.
\end{align*}

Combining the bounds\index{bounds} for $Q_{f}'$, $Q_{f}''$ and $Q_{hi}$, $i=1,\dots,6$,
we arrive at
\[
\int_0^t\! \frac{\abs{u(x,t) - \tilde u(x,t) - u(x,s)
+ \tilde u(x,s)}}{(t-s)^{\alpha+1}}ds \le C \xi\eta
\int_0^t\norm{u - \tilde u}_{\alpha,\infty,s}
(t-s)^{-2\alpha} ds.
\]
Taking suprema we get \eqref{eq:uniq-term2}.
\end{proof}

\begin{proof}[Proof of Theorem \ref{th:mild}]
Adding the bounds\index{bounds} from Lemmata \ref{l:uniq-1st}--\ref{l:uniq-2nd}, we
obtain that for all $t\in[0,T]$
\begin{align*}
\norm{u-\tilde u}_{\alpha,\infty,t} &\le C \xi\eta\int_0^t
\norm{u-\tilde u}_{\alpha,\infty,s} (t-s)^{-2\alpha} ds
\\
&\le C \xi\eta\, t^{2\alpha} \int_0^t
\norm{u-\tilde u}_{\alpha,\infty,s} (t-s)^{-2\alpha} s^{-2\alpha} ds.
\end{align*}
Then $\norm{u-\tilde u}_{\alpha,\infty,T} = 0$ by the generalized Gr\"
onwall lemma \cite[Lemma 7.6]{NR02}.
\end{proof}

\begin{acknowledgement}[title={Acknowledgments}]
The authors are grateful to Yuliya Mishura for giving the impetus to this
article and her constant advise and support.
\end{acknowledgement}

\begin{funding}
The first author acknowledges that the present research is carried through
within the frame and support of the ToppForsk project nr. \gnumber[refid=GS1]{274410} of the \gsponsor[id=GS1,sponsor-id=501100005416]{Research Council of Norway} with title STORM: Stochastics for Time-Space Risk Models.
\end{funding}

\end{document}